\documentclass[11pt,reqno]{amsart}
\usepackage{amssymb}
\usepackage{amsmath}
\usepackage{enumitem}
\usepackage{mathrsfs}
\usepackage[all]{xy}
\usepackage{hyperref}
\usepackage{marginnote}

\usepackage{tikz-cd}

\usepackage{stmaryrd}

\usepackage[margin=1.25in]{geometry}

\RequirePackage{mathrsfs} 

\newtheorem{theorem}{Theorem}[section]
\newtheorem{lemma}[theorem]{Lemma}
\newtheorem{proposition}[theorem]{Proposition}
\newtheorem{corollary}[theorem]{Corollary}

\theoremstyle{definition}
\newtheorem*{ack}{Acknowledgements}

\newtheorem{remark}[theorem]{Remark}

\newtheorem{definition}[theorem]{Definition}

\numberwithin{equation}{section} \numberwithin{figure}{section}

\DeclareMathOperator{\Pic}{Pic}

\DeclareMathOperator{\Spec}{Spec}

\DeclareMathOperator{\an}{an}

\newcommand{\Qbar}{\overline{\QQ}}

\newcommand\ZZ{\mathbb{Z}}

\newcommand\QQ{\mathbb{Q}}

\newcommand\CC{\mathbb{C}}

\newcommand\OO{\mathcal{O}}

\usepackage{color}

\definecolor{orange}{rgb}{1,0.5,0}

\author{Ariyan Javanpeykar}
\address{Ariyan Javanpeykar \\
	Institut f\"{u}r Mathematik\\
	Johannes Gutenberg-Universit\"{a}t Mainz\\
	Staudingerweg 9, 55099 Mainz\\
	Germany.}
\email{peykar@uni-mainz.de}

\author{Siddharth Mathur}
\address{Siddharth Mathur, CNRS, Universit\'e Paris-Saclay, Laboratoire de math\'ematiques d’Orsay, F-91405, Orsay, France\\}
\urladdr{https://sites.google.com/view/sidmathur/home} 



\title[Smooth hypersurfaces in abelian varieties]{Smooth hypersurfaces in abelian varieties over arithmetic rings}

\begin{document}

	\begin{abstract}   Let $A$ be an abelian scheme of   dimension at least four over a $\mathbb{Z}$-finitely generated integral domain $R$  of characteristic zero, and let $L$ be an ample line bundle on $A$.   We prove that the set of smooth hypersurfaces $D$ in $A$ representing $L$ is finite by showing that the moduli stack of such hypersurfaces has only finitely many $R$-points. We accomplish this by using level structures to interpolate finiteness results between this moduli stack and the stack of canonically polarized varieties. 
		\end{abstract}

	\maketitle

	\thispagestyle{empty}

	\section{Introduction}  
	Finiteness results are among the most celebrated achievements in arithmetic geometry. Consider the following striking examples:
	\begin{description} 
\item[Faltings] The set of isomorphism classes of  abelian schemes of dimension $g \geq 1$ over a fixed number ring is finite \cite{Faltings2}. In fact, the moduli stack of $g$-dimensional principally polarized abelian schemes $\mathcal{A}_g$ has only finitely many $\mathcal{O}_{K,S}$-points (up to isomorphism) for any number field $K$ and any fixed finite set of finite places $S$.
\item[Lawrence-Sawin] The set of isomorphism classes of smooth hypersurfaces $H$ in an abelian scheme $A$ of dimension $g \geq 4$ over a given number ring corresponding to a fixed ample class in $\mathrm{Pic}(A)$  is finite  \cite{lawrence2020shafarevich}. Once again, there is a stack, we call it $\mathcal{AH}_{g,d}$, which overparametrizes this set and which has only finitely many $\mathcal{O}_{K,S}$-points, up to isomorphism.
\end{description}
	
	 It is natural to ask if these results hold for larger fields and, in fact, Lang intimated that such statements over number fields should persist over finitely generated fields over $\mathbb{Q}$ (see \cite[pg. 202]{Lang2}). In fact, this \emph{persistence} is supported by a series of conjectures, due to Lang, which link arithmetic and complex geometry \cite{Lang2, JBook}. However, it is difficult to show that an arbitrary moduli stack admitting only finitely many $\mathcal{O}_{K,S}$-points should only admit finitely many $R$-points for any normal $\mathbb{Z}$-finitely generated domain of characteristic zero. We will call such rings $R$ \emph{arithmetic rings}; they may be viewed as higher dimensional analogues of number rings $\mathcal{O}_{K,S}$. Indeed, the number rings $\mathcal{O}_{K,S}$ are precisely the arithmetic rings of dimension one. Note that Faltings' result holds for arithmetic rings. The goal of this paper is to extend the result of Lawrence-Sawin to this setting.
	
	\begin{theorem}[Main Result]\label{thm1} Let $R$ be an arithmetic ring with fraction field $K$, and let $A$ be an abelian scheme of relative dimension at least four over $R$. Let $\mathcal{L} \in \mathrm{Pic}(A_K)$ be  ample. Then the set of $R$-smooth hypersurfaces $X\subset A$ such that $\mathcal{O}_{A_K}(X_K)$ is isomorphic to  $\mathcal{L}$ is finite. 
	\end{theorem}

	 The problem to overcome in proving Theorem \ref{thm1} is that the abelian scheme may not be defined over a number field (see \cite[Theorem 1.2]{JLitt} for this special case). As such, it is not clear if Lawrence-Sawin's theorem \cite{lawrence2020shafarevich} has arithmetic consequences over the larger fields $K$. Our idea is to introduce a geometric device, a moduli space defined over $\QQ$, to transport finiteness results over to arithmetic rings. After all, the geometry of this moduli space persists when passing to larger fields.

Instead of attempting to reproduce Lawrence-Sawin's arguments over larger fields, we use our work in \cite{JAut, JLalg, JSZ, JLM}  to give a conceptual proof of how to reduce Theorem \ref{thm1} to Lawrence-Sawin's main result \cite{lawrence2020shafarevich}. In fact, our strategy applies more generally to a stack $\mathcal{M}$ with finitely many $\mathcal{O}_{K,S}$-points for any number ring $\mathcal{O}_{K,S}$: \emph{if $\mathcal{M}$ additionally has a finite \'etale atlas and a quasi-finite map to the stack of canonically polarized varieties, then the set of isomorphism classes of $\mathcal{M}(R)$ is finite for all arithmetic rings $R$}.

The ideas of this paper fit into the  broader philosophy  of Shafarevich which says, roughly speaking, that  the set of objects of fixed type over a given arithmetic ring should be finite, under quite general conditions (see, for example,  \cite{Andre, Faltings2, JL, Scholl, Teppei}). This program can be made more precise by viewing objects as points on some geometric object  (such as a variety or stack) and appealing to the hyperbolicity of such a space in combination with Lang-Vojta's conjectures \cite{Lang2}.

Crucial to our proof is the moduli stack $\mathcal{AH}_{g,d}^{\text{sm}}$  classifying   pairs   $(A, P, D)$ with $A$ a $g$-dimensional  abelian variety, $P$ an $A$-torsor, and $D$ a smooth hypersurface in $P$ of  degree $d$;   see Definition \ref{defn} for a precise definition.  This stack is studied carefully in \cite{JLM} and,  to prove Theorem \ref{thm1}, it suffices to prove the finiteness of its $R$-points.

\raggedbottom

\subsection*{Outline of proof}
Combining the finiteness theorems of Lawrence-Sawin  and Faltings (see Theorem \ref{thm:sawinhyp}),  one sees that the  aforementioned stack $\mathcal{AH}_{g,d}^{\text{sm}}$    has only finitely many  points in any given number ring (see Theorem \ref{thm:sawinhyp}). Proving Theorem \ref{thm1}  amounts to establishing the \emph{Persistence Conjecture} (see  \cite[Conjecture~1.5]{JAut} and \cite[Remark~4.13]{JLalg}) for $\mathcal{AH}_{g,d}^{\text{sm}}$. To do so, we proceed in four steps.

\begin{enumerate}
\item We show that $\mathcal{AH}_{g,d}^{\text{sm}}$  is uniformisable over $\mathbb{Z}[1/2d]$ (see Proposition \ref{prop:uniff}), i.e., there is a finite \'etale surjective morphism $U \to \mathcal{AH}_{g,d,\mathbb{Z}[1/2d]}^{\text{sm}}$ with $U$ a $\mathbb{Z}[1/2d]$-scheme.    
\item  We show that    $U$   maps quasi-finitely to the stack of canonically polarized varieties (Corollary \ref{cor:qf}) by studying the deformation theory of smooth ample divisors in abelian varieties. 
\item Since $U$ maps quasi-finitely to the stack of canonically polarized varieties, we may    invoke   \cite[Theorem~1.4]{JSZ} to see that the variety $U$ satisfies the Persistence Conjecture. 
\item   To conclude, we descend the finiteness of $U(R)$ for every arithmetic ring $R$ to    our moduli stack $\mathcal{AH}_{g,d}^{\text{sm}}$   by appealing to the stacky Chevalley-Weil theorem \cite{JLalg} and the fact that $U\to \mathcal{AH}_{g,d,\mathbb{Q}}^{\text{sm}}$ is a finite \'etale cover.
\end{enumerate}

In summary,   the proof of Theorem \ref{thm1} is   obtained by combining Lawrence-Sawin's finiteness result over number fields \cite{lawrence2020shafarevich}  with a careful study of $\mathcal{AH}_{g,d}^{\text{sm}}$ and the work done in \cite{JLalg, JAut, JLM, JSZ}. Thus, this article also illustrates the utility of these latter works. Moreover, the methods of this paper can be applied in wider generality.  For example, we expect that Lawrence-Sawin's results can be extended to certain complete intersections in an abelian variety over a number ring. Then the methods of the present article should imply the finiteness of such complete intersections over arithmetic rings.

As is apparent from the above outline, the uniformisability of the stack $\mathcal{AH}_{g,d}^{\text{sm}}$ plays a central role in our proof. Note that many smooth Deligne-Mumford stacks do not admit finite \'etale atlases! Indeed, such an atlas exists precisely when the stack can be written as a (stack) quotient of an algebraic space by a finite group. For instance, none of the weighted projective lines $\mathcal{P}(1,n)$ for $n>1$ enjoy this property. Moreover, the uniformisability of a given (moduli) stack is of independent interest; see, for example, the related question of Fulton (partially) addressed in \cite{geraschenko2015torus} and a daring conjecture \cite[Conjecture~1.6]{JSZ} concerning the uniformisability of the stack  $\mathcal{CP}$ of   canonically polarized varieties.

We offer two proofs of the uniformisability of $\mathcal{AH}_{g,d, \mathbb{Q}}$. The first combines two types of level structure to yield a uniformisation over $\mathbb{Z}[1/2d]$. The second proof uses transcendental methods and   only works over $\mathbb{Q}$. On the other hand, it shows that the (open) locus in $\mathcal{CP}$ consisting of varieties which embed into an abelian variety is uniformisable, a fact that will be useful in later work. 

	\begin{ack}    
We would like to acknowledge Will Sawin for several useful discussions which ultimately led to our study of the stack $\mathcal{AH}_{g,d}^{\text{sm}}$. The proof of Lemma \ref{lem:sawin} is due to him. We thank Daniel Loughran for further discussions on $\mathcal{AH}_{g,d}^{\text{sm}}$.  The first-named author gratefully acknowledges the IHES for its hospitality. The second-named author is supported by the European Research Council (ERC) under the European Union’s Horizon 2020 research and innovation programme (grant agreement No. 851146).
	\end{ack}

	\section{Abelian schemes, their torsors and the Albanese}
	
	Before introducing the main object of interest, we will introduce a few necessary facts about abelian schemes, their principal homogenous spaces, and the Albanese morphism. Recall that a morphism $\pi: A \to S$ of schemes is said to be an \emph{abelian scheme (over $S$)} if it is a smooth proper group scheme with geometrically connected fibers. If $P \to S$ is a smooth proper morphism of algebraic spaces whose geometric fibers admit the structure of an abelian variety, then we say that $P$ is a \emph{para-abelian} algebraic space (see \cite[VI, Theorem 3.3]{FGA} and \cite[Definition 4.2]{https://doi.org/10.48550/arxiv.2101.10829}). Torsors under an abelian scheme furnish examples of para-abelian algebraic spaces and, in fact, these are the only examples. 
	
	Indeed, the theory of para-abelian spaces is equivalent to the theory of torsors under an abelian scheme, as we will now make precise. Let $\mathscr{P}$ denote the fibered category whose objects over an $S$-scheme are pairs $(A \to T, P \to T)$ where $A$ is an abelian scheme and $P$ is an $A$-torsor. The morphisms in $\mathscr{P}_S$ are pairs of maps of $S$-schemes, $f: A \to A'$ and $g: P \to P'$ where $g$ is equivariant with respect to $f$. Now let $\mathscr{P}'$ denote the fibered category whose objects over an $S$-scheme consists of para-abelian algebraic spaces $P \to S$. Morphisms in $\mathscr{P}'_S$ are morphisms of $S$-schemes. 
		  
\begin{proposition} \label{prop:abtorsor} The forgetful map $F: \mathscr{P} \to \mathscr{P}'$ is an equivalence of fibered categories. \end{proposition}

\begin{proof} We will define an inverse functor $H: \mathscr{P}' \to \mathscr{P}$. If $P \to T$ is an object of $\mathscr{P}'_T$, by \cite[Theorem~5.3]{https://doi.org/10.48550/arxiv.2101.10829}, it follows that the group subscheme $G \subset \underline{\text{Aut}}_{P/T}$ consisting of those $T$-automorphisms $\sigma: P \to P$ with the property that the induced map $\sigma^*:\text{Pic}^0_{P/T} \to \text{Pic}^0_{P/T}$ is the identity, is an abelian scheme and that $P \to T$ is a torsor under $G$. We define $H(P \to T)=(G \to T, P \to T)$.  By \cite[Proposition~5.4]{https://doi.org/10.48550/arxiv.2101.10829}, $H$ is a functor.   It remains to show that, for every $A$-torsor $P$,  there is a natural isomorphism $i: A \to G$ (where $G \subset \underline{\text{Aut}}_{P/T}$ is defined as above) such that the identity $P \to P$ is equivariant for $i$.

Since $A$ acts simply transitively on $P$, there is a given immersion $A \to \underline{\text{Aut}}_{P/T}$. For any $a \in A(T)$, we call the associated automorphism $t_a: P \to P$. Note that it induces an isomorphism $t_a^*: \text{Pic}^0_{P/T} \to \text{Pic}^0_{P/T}$ by pulling back line bundles and, $t_a^*$ necessarily preserves the $n$-torsion for every $n\geq 1$. In other words,  for every $n\geq 1$, there is a natural morphism of group schemes
\[\phi_n: A \to \underline{\text{Aut}}(\text{Pic}^0_{P/T}[n])\]
and, since the latter is affine over $S$ (see, for example, \cite[Lemma 4.1]{LorenziniSchroer}) and $A$ is proper, each $\phi_n$ must be trivial. Since the union of these finite subgroups is schematically dense in $\text{Pic}^0_{P/T}$, this implies $A \subset G \subset \underline{\text{Aut}}_{P/T}$ and since they are both abelian schemes of the same dimension, they are equal as subschemes of $\underline{\text{Aut}}_{P/T}$, as desired. \end{proof}

It follows that when studying torsors under abelian schemes, one may forget the abelian scheme and the accompanying action without losing any essential information.

\begin{theorem} If $X \to S$ is a smooth proper morphism of schemes over $ \mathbb{Q}$, then there is a para-abelian algebraic space $\mathrm{Alb}_{X/S} \to S$ such that:
\begin{enumerate} \item There is an $S$-morphism $X \to \mathrm{Alb}_{X/S}$ which is universal for all maps from $X$ to para-abelian spaces, and 
\item The formation of the algebraic space $\mathrm{Alb}_{X/S} \to S$ and the morphism $X \to \mathrm{Alb}_{X/S}$ are compatible with arbitrary base change on $S$. \end{enumerate} \end{theorem}

\begin{proof} Combine \cite[VI, Theorem 3.3]{FGA} and \cite[Proposition 5.20, Remark 5.21]{fantechi2005fundamental}.\end{proof}

\noindent We call $\text{Alb}_{X/S}$ the \emph{Albanese} associated to $X/S$ and note that it may not have the structure of an abelian scheme, it is merely a \emph{torsor} under one. It admits a group structure if and only if it has an $S$-point, and in this case, the structure is uniquely determined by the choice of $S$-point. 

If $P$ is para-abelian, the functor $\Pic_{P/S}^{0}$ (see \cite[Proposition~2.1.3]{abelianolsson} and \cite[Remark 1.5]{faltingschai}) is representable by an abelian scheme over $S$. Note that if $A$ is an abelian scheme over $S$, then   $A^{\vee}$ will denote the abelian scheme representing $\Pic_{A/S}^{0}$ and is called the \emph{dual} of $A$. Moreover, for every    torsor  $P/S$ under $A/S$, there is a natural identification $A^{\vee} \cong \Pic^0_{P/S}$ (see \cite[Proposition XIII, 1.1 (ii)]{raynaud2006faisceaux}).

	 As in \cite[Definition~2.1.2]{abelianolsson}, a \emph{degree d polarization} on $A$ is a finite flat morphism $\lambda: A \to A^{\vee}$ of group schemes whose kernel is finite locally free over $S$ of degree $d^2$. Given a relatively ample line bundle $L$ on an abelian scheme $A/S$, we say \emph{L has degree d} if the morphism
\[\lambda_L: A \to A^{\vee}, \quad a \mapsto t_a^*L \otimes L^{\vee},\]
is a degree $d$ polarization. Moreover, given a relatively ample line bundle $L$ on a torsor $\pi: P \to S$ under an abelian scheme $A/S$, the sheaf $\pi_*L$ is locally free of some rank $d$ and $L$ induces a polarization $\lambda_{L}: A \to \Pic^0_{P/S}=A^{\vee}$ of degree $d$ (see \cite[2.2.3, 2.2.4]{olssonabelian2}), so in this case we also say \emph{L has degree d}. 

\section{The stack of smooth abelian hypersurfaces}

\begin{definition}\label{defn} Let $\mathcal{AH}_{g,d}$ denote the fibred category over the category of schemes whose fibre over a scheme $S$ is the groupoid of triples $(\pi: P \to S, L, s:\mathcal{O}_P \to L)$ such that 
\begin{enumerate} 
\item $\pi: P \to S$ is a para-abelian algebraic space of relative dimension $g$.
\item $L  $ is a relatively ample line bundle  on $\pi:P\to S$ of degree  $d$ on each geometric fibre of $\pi$.
\item the zero locus $V(s) \subset P$ is flat over $S$. 
\end{enumerate} 
A morphism $(\pi': P' \to S', L', s':\mathcal{O}_{P'} \to L') \to (\pi: P \to S, L, s:\mathcal{O}_{P} \to L)$ over $S' \to S$ consists of a morphism $f$ which makes the following square Cartesian

\begin{center}
\begin{tikzcd}
 P' \arrow[d] \arrow[r, "f"] & P \arrow[d]  \\
 S' \arrow[r] & S 
\end{tikzcd}
\end{center}
and an isomorphism $g: f^*L \simeq L'$ which sends $f^*s$ to $s'$. We call $\mathcal{AH}_{g,d}$ the \emph{moduli stack of abelian hypersurfaces of degree d}. The subcategory $\mathcal{AH}_{g,d}^{\text{sm}} \subset \mathcal{AH}_{g,d}$ consisting of $(\pi: P \to S, L, s:\mathcal{O}_P \to L)$ where $V(s) \subset P$ is smooth over $S$ will be referred to as the \emph{moduli stack of smooth abelian hypersurfaces of degree d}. Similarly, we define  the fibred category  $\overline{\mathcal{AH}}_{g,d}$  whose fibre over a scheme $S$ is the groupoid of pairs $(\pi:P\to S, L)$ with $P$ as in $(1)$ and $L$ as in $(2)$.
\end{definition}
 
\begin{proposition}[Basic properties]\label{prop:basic_properties} The stack  $\mathcal{AH}_{g,d}^{\text{sm}}$   is a finite type   algebraic stack with finite  diagonal over $\mathbb{Z}$.
\end{proposition}
\begin{proof}
This is proven in \cite[\S 6]{JLM}. Indeed, the algebraicity  (resp. finiteness of the diagonal) follows from \cite[Proposition~6.3]{JLM} (resp. \cite[Propostion~6.4]{JLM}).
 \end{proof}

	An algebraic  stack   $X$ is \emph{uniformisable (by an algebraic space)} if there exists an algebraic   space $U$ and a finite \'etale   surjective morphism $U\to X$; see \cite[Definition 6.1]{Noohi}. By \cite[Th\'eor\`eme.~6.1]{LMB}, an algebraic stack $X$ is uniformisable if and only if there exists a finite (abstract) group $G$, an algebraic space $U$,  and an action of $G$ on $U$ such that $X\cong [U/G]$.  
	
 	We will construct an explicit uniformisation of the stack $\mathcal{AH}^{\text{sm}}_{g,d,\mathbb{Z}[1/2d]}$. Our proof will require a well-known lemma that we could not locate in the literature. As such, we include the statement and proof below. Recall that a flat affine group scheme $G/S$ is said to be \emph{linearly reductive} if the structure map $\pi: BG \to S$ has the property that $\pi_*:\text{QCoh}(BG) \to \text{QCoh}(S)$ is exact, where $BG:=[S/G]$ denotes the classifying stack of $G$-torsors over $S$ (see, for example, \cite[Tag~0CQJ]{stacks-project}).
	
\begin{lemma}\label{lemma:isom_groups} Let $G$ and $H$ be smooth linearly reductive group schemes over a scheme $S$ such that, for every geometric point $\bar{s}$ of $S$, there is an isomorphism  $G_{\bar{s}} \to H_{\bar{s}}$ of group schemes. Then,   there is an \'etale surjective morphism $S' \to S$ and an isomorphism $G_{S'} \to H_{S'}$. 
\end{lemma}

\begin{proof}  Consider the functor of (group) isomorphisms $I=\underline{\text{Isom}}_S(G,H) \to S$. Since $G$ and $H$ are linearly reductive over $S$,   it follows from  \cite[Theorem~2]{https://doi.org/10.48550/arxiv.2101.12460} and \cite[Remarque I.1.7.3]{zbMATH05948488} that  $I$ is representable by a scheme which is locally of finite presentation over $S$. Moreover, $I \to S$ is a smooth morphism. Indeed, if $A' \to A$ is a surjective morphism of affine schemes with nilpotent ideal $I$, the obstruction to extending an isomorphism $\phi_A: G_A \to H_A$ to a morphism over $\Spec A$ to one over $\Spec A'$ lives in the Hochschild cohomology group $H^2(G_A, \text{Lie}_{H_A/A} \otimes I)$ (see \cite[Corollaire~III.2.6]{zbMATH05948488}). This cohomology group vanishes because $G$ is linearly reductive, and the lifted morphism $\phi_{A'}$ is an isomorphism by \cite[Corollaire~17.9.5]{zbMATH03245973}. Thus, the morphism $I \to S$ is smooth and surjective. In particular,  the morphism $I\to S$   admits sections \'etale locally by \cite[Corollaire~17.16.3 (ii)]{zbMATH03245973}, as desired. 
\end{proof}	
	
	We construct a finite \'etale atlas of $\mathcal{AH}_{g,d, \mathbb{Z}[1/2d]}^{\text{sm}}$ by introducing extra structure to the moduli problem to ensure the resulting moduli stack has trivial stabilizers.  Our construction combines moduli problems introduced by Mumford in \cite{MR204427} and developed further by Olsson in \cite{abelianolsson}. 
	
	\begin{remark}\label{remark:tgd}
By Proposition \ref{prop:abtorsor}, the stack $\overline{\mathcal{AH}}_{g,d}$ is equivalent to the stack $\mathcal{T}_{g,d}$ defined in \cite[Section~5.1.1]{abelianolsson} which parametrizes triples $(A,P,L)$ with $A$ an abelian scheme, $P$ an $A$-torsor, and $L$ a relatively ample degree $d$ line bundle on $P$.
\end{remark}
	
\begin{remark} \label{preambleuni} The stack $\mathcal{AH}_{g,d}^{\text{sm}}$ is  closely related to the moduli stacks $\mathcal{T}_{g,d}$ (Remark \ref{remark:tgd}) and $\mathcal{A}_{g,d}$, where the stack $\mathcal{A}_{g,d} $ parametrizes abelian schemes with a fixed polarization of degree $d$; see \cite[Chapter~5]{abelianolsson}.   
 Indeed, the stack $\mathcal{AH}_{g,d}^{\text{sm}}$ is naturally an open substack of $\mathbb{V}( (\pi_{u})_{\ast} \mathcal{L})$, where $\pi_{u}: (\mathcal{A}, \mathcal{P},  \mathcal{L})\to \mathcal{T}_{g,d}$ is the universal object of $\mathcal{T}_{g,d}$. Thus, we have the following morphisms
 \[
 \mathcal{AH}_{g,d}^{\text{sm}}\subset \mathbb{V}((\pi_{u})_{\ast} \mathcal{L}) \longrightarrow \mathcal{T}_{g,d} \longrightarrow \mathcal{A}_{g,d}.
 \]
 where the first inclusion is open and the rightmost arrow is defined by sending $(A,P,L)$ to $(A,\lambda_L)$ with $\lambda_L: A\to A^{\vee}$   the associated degree $d$ polarization. In fact, this rightmost arrow is a gerbe by \cite[Proposition 5.1.4]{abelianolsson}. Although the stack $\mathcal{A}_{g,d}$ is uniformisable (see \cite[Theorem~7.9]{GIT}), the pull-back of a uniformisation of $\mathcal{A}_{g,d}$ does not always give a uniformisation of $\mathcal{AH}_{g,d}^{\text{sm}}$ since $\mathcal{AH}_{g,d} \to \mathcal{A}_{g,d}$ is not necessarily representable. However, as we aim to show below, one can construct an auxiliary moduli problem to remedy this issue.
\end{remark}

\begin{proposition}\label{prop:uniff} The stack  $\mathcal{AH}_{g,d, \mathbb{Z}[1/2d]}^{\text{sm}}$ is uniformisable.
\end{proposition}

\begin{proof} By definition, the relative inertia group $\mathcal{G}_{(A,P,L)}$ of the morphism $\mathcal{T}_{g,d} \longrightarrow \mathcal{A}_{g,d}$ at an object $(A,P,L) \in \mathcal{T}_{g,d}(S)$ is the kernel of the homomorphism
\[ \underline{\text{Aut}}_{\mathcal{T}_{g,d}}(A,P,L) \to \underline{\text{Aut}}_{\mathcal{A}_{g,d}}(A,\lambda_L).\]
Thus, $\mathcal{G}_{(A,P,L)}(S)$ is  the group of pairs $(f,g)$, where $f:P \to P$ is an $A$-equivariant automorphism and $g:f^*L \to L$ is an isomorphism. However, since $f$ respects the action of the abelian scheme $A$, the morphism $f$ is equal to translation by a (unique) point $x \in A(S)$. Furthermore,  since we have $t_x^*L \cong L$, it follows that $x \in H(L)(S)$, where $H(L)=\text{Ker}(\lambda_L: A \to A^{\vee}=\Pic^0_{P/S})$. In fact, we have an exact sequence of group sheaves
\[1 \to \mathbb{G}_m \xrightarrow{\iota} \mathcal{G}_{(A,P,L)} \to H(L) \to 1,\]
where $c \in \mathbb{G}_m(S)$ gets mapped to $(\text{id}_P,c)$ (see \cite[5.1.3]{abelianolsson}). Note that the subgroup $$\underline{\text{Aut}}_{\mathcal{AH}_{g,d}}(A,P,L,s) \subset \underline{\text{Aut}}_{\mathcal{T}_{g,d}}(A,P,L)$$ intersects $\iota(\mathbb{G}_m)$ trivially, as  units scale the section $s$. 

To prove the proposition, it suffices to find a finite \'etale cover $\psi: \Sigma_{g,d} \to \mathcal{T}_{g,d,\mathbb{Z}[1/2d]}$ such that, for every $p \in \Sigma_{g,d}(S)$, the induced subgroup scheme 
\[\underline{\text{Aut}}_{\Sigma_{g,d}}(p) \hookrightarrow \underline{\text{Aut}}_{\mathcal{T}_{g,d, \mathbb{Z}[1/2d]}}(A,P,L)\] intersects $\mathcal{G}_{(A,P,L)}$ in $\iota(\mathbb{G}_m)$ (with $\psi(p)=(A,P,L))$. Indeed, in this case the pullback $$\overline{\Sigma}_{g,d}=\Sigma_{g,d} \times_{\mathcal{T}_{g,d, \mathbb{Z}[1/2d]}} \mathcal{AH}_{g,d, \mathbb{Z}[1/2d]}^{\text{sm}}$$ has automorphism groups which meet the various $\mathcal{G}_{(A,P,L)}$ trivially.
Thus, $\overline{\Sigma}_{g,d}$ is a finite \'etale cover of $\mathcal{AH}_{g,d, \mathbb{Z}[1/2d]}^{\text{sm}}$ and is representable \emph{over} the stack $\mathcal{A}_{g,d}$. Now, let $\mathcal{A}_{g,d,n} \to \mathcal{A}_{g,d}$ denote the moduli stack of degree $d$ polarized abelian schemes with level $n$ structure, and note that $\mathcal{A}_{g,d,n}$ is a scheme and that the morphism is finite, \'etale and surjective  (see, e.g., the proof of \cite[Theorem~2.1.11]{olssonabelian2} and \cite[Theorem 7.9]{GIT}). It follows that $\mathcal{A}_{g,d,n} \times_{\mathcal{A}_{g,d}} \overline{\Sigma}_{g,d} \to \mathcal{AH}^{\text{sm}}_{g,d,\mathbb{Z}[1/2d]}$ is a uniformisation. 

It remains to construct the auxiliary moduli stack $\Sigma_{g,d}$. Fix a sequence of $g$ positive integers $\delta=(d_1,...,d_g)$ where $d_i\mid d_{i+1}$ with $d=d_1\cdots d_g$ (we call $\delta$ a \emph{type}) and define the group scheme $G(\delta)$ over $S$, a $\mathbb{Z}[1/2d]$-scheme, as follows.  First, its underlying scheme is $\mathbb{G}_{m} \times K(\delta) \times K(\delta)^{\vee}$, where $K(\delta)=\bigoplus_{i=1}^g \mathbb{Z}/d_i\mathbb{Z}$ and $K(\delta)^{\vee}$ denotes the Cartier dual of $K(\delta)$, i.e. $K(\delta)^{\vee}= \bigoplus_{i=1}^g \mu_{d_i}$. Then, the group law is defined to be
\[(\alpha, x,l)(\alpha',x',l')=(\alpha\alpha'l'(x),xx',ll').\] This  shows that there is an exact sequence
\[1 \to \mathbb{G}_m \to G(\delta) \to H(\delta) \to 1,\]
where $H(\delta)=K(\delta) \times K(\delta)^{\vee}$. In fact, given any $(A,P,L) \in \mathcal{T}_{g,d}(S)$, we see that \'etale locally on $S$ there is a type $\delta$ and an isomorphism of groups $\phi$ which make the following diagram commute

\begin{center}
\begin{tikzcd}
 1 \arrow[r] & \mathbb{G}_m \arrow[d,"\textrm{id}"] \arrow[r] & \mathcal{G}_{(A,P,L)} \arrow[d,"\phi"] \arrow[r]  & H(L) \arrow[r] \arrow[d] & 1 \\
 1 \arrow[r] & \mathbb{G}_m \arrow[r] & G(\delta) \arrow[r]  & H(\delta) \arrow[r] & 1.
\end{tikzcd}
\end{center}
When $S$ is an algebraically closed field, this follows from \cite[pg. 294-295]{MR204427}. To prove this for a general scheme $S$,  note that the type $\delta$ for each $\mathcal{G}_{(A,P,L)}$ is locally constant on $S$, so that  we may assume $S$ is connected and that $\delta$ is constant. Thus, there are isomorphisms $\phi$ as above at every geometric point of $S$. Therefore, since $G(\delta)$ and $\mathcal{G}_{(A,P,L)}$ are   linearly reductive by \cite[Proposition~12.17]{goodmodulispaces}, there exists   an isomorphism \'etale locally on $S$ by Lemma \ref{lemma:isom_groups}.

The fact that isomorphisms exist \'etale locally implies that isomorphisms inducing the identity on $\mathbb{G}_m$ also exist \'etale locally. Indeed,  any isomorphism preserves the connected component and hence induces an automorphism of $\mathbb{G}_m$, so that we get a morphism \[\underline{\text{Isom}}(G_{(A,P,L)},G(\delta)) \to \underline{\text{Aut}}(\mathbb{G}_m)=\mathbb{Z}/2\mathbb{Z}.\]
Let $I_{\mathbb{G}_m}$ be its (open) fiber over the identity section, and note that   $I_{\mathbb{G}_m}$ inherits smoothness from $I$. Surjectivity of $I_{\mathbb{G}_m} \to S$ follows from the case when $S$ is a field (see \cite[pg. 294-295]{MR204427}). Thus, it follows that $I_{\mathbb{G}_m}$ admits points \'etale locally on $S$, i.e., isomorphisms $\phi$ which induce the identity of $\mathbb{G}_m$ exist \'etale locally.

Consider the open and closed substack $\mathcal{T}_{g,\delta} \subset \mathcal{T}_{g,d, \mathbb{Z}[1/2d]}$ consisting of those $(A,P,L)$ with $H(L)$ having type $\delta$. Following \cite[6.3.22]{abelianolsson}, we define $\Sigma_{g,\delta} \to \Spec \mathbb{Z}[1/2d]$ to be the stack of tuples whose fiber over $S$ is $(A,P,L, \sigma)$ where $(A,P,L) \in \mathcal{T}_{g,\delta}(S)$ and $\sigma:\mathcal{G}_{(A,P,L)} \to G(\delta)$ is an isomorphism of group schemes which is the identity over the respective copies of $\mathbb{G}_m$. A morphism $(A,P,L, \sigma) \to (A',P',L', \sigma')$ in $\Sigma_{g,\delta}(S)$ is a map $\eta: (A,P,L) \to (A',P',L')$ in $\mathcal{T}_{g,\delta}(S)$ such that 

\begin{center}\begin{tikzcd}
G(\delta) \arrow[rd, "\sigma'"] \arrow[r, "\sigma"] & \mathcal{G}_{(A,P,L)} \arrow[d, "\eta(-)\eta^{-1}"] \\
& \mathcal{G}_{(A',P',L')} 
\end{tikzcd} \end{center}

\noindent commutes. 

The morphism $\Sigma_{g, \delta} \to \mathcal{T}_{g,\delta}$  defined by forgetting $\sigma$ is finite, \'etale, and surjective. Indeed, the argument above shows that the fiber product of this morphism along any map $S \to \mathcal{T}_{g,\delta}$ is a torsor under the group of automorphisms $\underline{\text{Aut}}_{\mathbb{G}_m}(G(\delta))$ of $G(\delta)$ which act as the identity on $\mathbb{G}_m$. By \cite[Proposition 6.3.7]{abelianolsson} it follows that the morphism is finite and \'etale. Since $\Sigma_{d,\delta}\to \mathcal{T}_{g,\delta}$ is surjective, it remains to show that 
\[\underline{\text{Aut}}_{\Sigma_{g,\delta}}(A,P,L, \sigma) \subset \underline{\text{Aut}}_{\mathcal{T}_{g,d,\mathbb{Z}[1/2d]}}(A,P,L)\]
 intersects $\mathcal{G}_{(A,P,L)}$ in $\iota(\mathbb{G}_m)$. If $g$ is a point of 
 \[\underline{\text{Aut}}_{\Sigma_{g,\delta}}(A,P,L, \sigma) \times_{\underline{\text{Aut}}_{\mathcal{T}_{g,d,\mathbb{Z}[1/2d]}(A,P,L)}} \mathcal{G}_{(A,P,L)}\]
 then  conjugation by $g$ preserves the isomorphism $\sigma: \mathcal{G}_{(A,P,L)} \to G(\delta)$. In other words, $g$ defines a central element of $\mathcal{G}_{(A,P,L)}$. Since the center of this group is $\iota(\mathbb{G}_m)$, this completes the proof.  \end{proof}

\begin{remark} \label{rem:olsson} The statement that the morphism $\Sigma_{g, \delta} \to \mathcal{T}_{g,\delta}$ defined by forgetting $\sigma$ is finite and \'etale is stated in \cite[Proposition 6.3.23]{abelianolsson}. However, we were unable to find a reference for the fact that the isomorphisms $\sigma: \mathcal{G}_{(A,P,L)} \to G(\delta)$ exist fppf locally. As such, we included this argument in the proof above to ensure completeness. \end{remark}

\section{Quasi-finiteness of the forgetful functor}\label{section3}
 
Let $\mathcal{CP}$  be the stack of smooth proper canonically polarized varieties over $\mathbb{Q}$, i.e., for $S$ a scheme over $\mathbb{Q}$, the objects of the groupoid $\mathcal{CP}(S)$ are smooth proper morphisms $X\to S$ of schemes whose geometric fibres are connected with ample canonical bundle, morphisms in $\mathcal{CP}(S)$ are $S$-isomorphisms of schemes.
For $h \in \mathbb{Q}[t]$ a polynomial, we let $\mathcal{CP}_h$ be the substack of smooth proper canonically polarized varieties with Hilbert polynomial $h$ (where the Hilbert  polynomial of $X\to S$ is computed with respect to $\omega_{X/S}$). Note that  $\mathcal{CP}_h$ is an open and closed substack of $\mathcal{CP}$, and that    $\mathcal{CP}$  is the (countable) disjoint union of the stacks $\mathcal{CP}_h$, where $h$ runs over $\mathbb{Q}[t]$. For each polynomial $h$,  the stack $\mathcal{CP}_h$ is a finite type separated Deligne-Mumford algebraic stack over $\mathbb{Q}$ with a quasi-projective coarse space; see \cite{Viehweg06}. 

Throughout this section, we let $g\geq 2$ be an integer and let $d$ be a positive integer. In particular, if $A$ is an abelian variety over a field $k$ of dimension $g$ and $D\subset A$ is a smooth ample hypersurface whose associated line bundle has degree $d$, then $\omega_{D/k}$ is ample by adjunction and $D$ is connected. Therefore, 
there is a functor  
\[ 
\mathcal{AH}_{g,d, \mathbb{Q}}^{\text{sm}} \to \mathcal{CP}, \quad (P,L,s) \mapsto \mathrm{V}(s)
\]   
which sends a tuple $(P, L,s)$ to the canonically polarized scheme   $\mathrm{V}(s)$ (and forgets the embedding $\mathrm{V}(s)\subset P$) for $g \geq 2$.

\begin{lemma}\label{lem:reps} The morphism  \[ 
\mathcal{AH}_{g,d, \mathbb{Q}}^{\text{sm}} \to \mathcal{CP}, \quad (P,L,s) \mapsto \mathrm{V}(s)
\] is representable.
\end{lemma}
\begin{proof}  
By \cite[Tag 04YY]{stacks-project},  we need to show  that the relative inertia stack 
\[\mu: I_{\mathcal{AH}_{g,d, \mathbb{Q}}^{\text{sm}}/\mathcal{CP}} \to \mathcal{AH}_{g,d, \mathbb{Q}}^{\text{sm}}\]
is a trivial group stack. The morphism $\mu$ is finite, as $\mathcal{AH}_{g,d, \mathbb{Q}}^{\text{sm}}$ and $\mathcal{CP}$ both have finite diagonals. Moreover, $\mu$ is unramified because its geometric fibers are finite reduced group schemes by Cartier's theorem (see \cite[Tag 047O]{stacks-project}). Thus, by \cite[Tag 04DG]{stacks-project}, if we show the geometric fibers of $\mu$ are all singletons then it would be a closed immersion and because $\mu$ admits a section it would necessarily be an isomorphism.  On the other hand,  if $x: \Spec \bar{k} \to \mathcal{AH}_{g,d, \mathbb{Q}}^{\text{sm}}$ is a $\bar{k}$-valued point, the fiber over $\mu$ is isomorphic to the kernel of the morphism of $\bar{k}$-group schemes $\text{Aut}_{\mathcal{AH}_{g,d, \mathbb{Q}}^{\text{sm}}}(x) \to \text{Aut}_{\mathcal{CP}}(f(x))$ (see, for example, the second diagram in the proof of \cite[Tag 050Q]{stacks-project}).  Thus,  we see that  it suffices to show that the kernel $K$ of the natural homomorphism from the inertia group of a $\bar{k}$-object $(A,P,L,s)$ of $\mathcal{AH}_{g,d,\mathbb{Q}}^{\text{sm}}$ to the automorphism group $\mathrm{Aut}(D)$ of $D:=\mathrm{V}(s)$ is trivial.  

To do so, we fix an isomorphism $P\cong A$ such that the hypersurface $D\subset P\cong A$ contains the origin of $A$. This can be done because $\bar{k}$ is algebraically closed. Then, an element of the above kernel $K$ corresponds to a homomorphism $f:A\to A$  with  $f(x) = x$ for every $x$ in $D$.   Since the kernel of the homomorphism $f-\mathrm{id}_A$ contains the ample divisor $D$ and     the smallest abelian subvariety containing $D$ is $A$ (since $g\geq 2$ and $\omega_{D/\mathbb{C}}$ is ample) the kernel of $f-\mathrm{id}_A$ equals $A$, so that    $f=\mathrm{id}_A$.  This shows  that  the kernel $K$ is trivial, and  concludes the proof.
\end{proof}

If $k$ is a field of characteristic zero and   $(A,P,L,s)$ is a $k$-object of $\mathcal{AH}_{g,d,\mathbb{Q}}^{\text{sm}}$, we let \[\mathrm{T}_{\mathcal{AH}_{g,d,\mathbb{Q}}^{\text{sm}}}((A,P,L,s))\] 	 denote the tangent space to $(A,P,L,s)$ in the stack  	$\mathcal{AH}_{g,d,\mathbb{Q}}^{\text{sm}}$. Similarly, for $D$ a $k$-object of $\mathcal{CP}$, we let $\mathrm{T}_{\mathcal{CP}}(D)$ be the tangent space of the object $D$ to $\mathcal{CP}$, and we note  that $\mathrm{T}_{\mathcal{CP}}(D)$ equals  $\mathrm{H}^1(D,\mathrm{T}_D)$.
The morphism \[ 
\mathcal{AH}_{g,d, \mathbb{Q}}^{\text{sm}} \to \mathcal{CP}, \quad (P,L,s) \mapsto \mathrm{V}(s)
\] defined above induces for each $k$-object $(A,P,L,s)$ of $
\mathcal{AH}_{g,d, \mathbb{Q}}^{\text{sm}} $  a  morphism of $k$-vector spaces
 \begin{eqnarray}\label{tan}
 \mathrm{T}_{\mathcal{AH}_{g,d,\mathbb{Q}}^{\text{sm}}}((A,P,L,s))\to  \mathrm{T}_{\mathcal{CP}}(\mathrm{V}(s)).
\end{eqnarray}
\begin{proposition}\label{prop:tangent} Let $k$ be a field of characteristic zero and let $(A,P,L,s)$ be a $k$-object of 
$\mathcal{AH}_{g,d, \mathbb{Q}}^{\text{sm}}$.  Then the following statements hold.
\begin{enumerate}
\item  The map (\ref{tan}) is injective.
\item If $g\geq 3$, then the map  (\ref{tan}) is an isomorphism. 
\end{enumerate}
\end{proposition}
\begin{proof}
Let $D:=\mathrm{V}(s)$, and note that $\iota:D \to  P$ is a smooth ample hypersurface. We let $ N_{D/P}$ be the normal bundle of $\iota:D\subset P$ on $D$.  By \cite[Proposition~3.4.17]{sernesidef}, the tangent space $
\mathrm{T}_{\mathcal{AH}_{g,d,\mathbb{Q}}^{\text{sm}}}((A,P,L,s))$ is naturally identified with $\mathrm{H}^1(P, T_P\langle D\rangle)$, where $T_P\langle D \rangle$ is defined to be the kernel in the short exact sequence
\[
0\to T_P\langle D \rangle \to T_P \to \iota_\ast N_{D/P} \to 0.
\]  Similarly,  the tangent space  $\mathrm{T}_{\mathcal{CP}}(D)$ can be identified with $\mathrm{H}^1(D,T_D)$. We set $T_P(-D) = T_P \otimes \mathcal{O}_P(-D)$, then the morphism  
\[
\mathrm{T}_{\mathcal{AH}_{g,d,\mathbb{Q}}^{\text{sm}}}((A,P,L,s))\to  \mathrm{T}_{\mathcal{CP}}(\mathrm{V}(s)) \]  on tangent spaces is given by $\mathrm{H}^1(d)$ induced by the diagram
\[
\xymatrix{     & &  0 \ar[d]  & &  0 \ar[d] & &   &     \\ 
  0 \ar[rr] & &  T_P(-D)\  \ar[d]  \ar[rr] & &  T_P(-D) \ar[d] \ar[rr] & & 0     \ar[d] &       \\
   0 \ar[rr] & &  T_P\langle D \rangle    \ar[d]_{d}  \ar[rr] & &  T_P \ar[d] \ar[rr] & & \iota_\ast N_{D/P}  \ar[r]   \ar[d] &  0     \\
    0 \ar[rr] & &  \iota_\ast T_D    \ar[d] \ar[rr] & &  \iota_\ast T_P|_D \ar[d] \ar[rr] & & \iota_\ast N_{D/P} \ar[r] \ar[d] &  0    \\
        & &   0  & &   0  & &  0   &      }
\]
Since the above diagram induces, for every $i\geq 1$,   an exact sequence
\[\xymatrix{
\mathrm{H}^i(P, T_P(-D)) \ar[r]  &  \mathrm{H}^{i}(P, T_P\langle D \rangle )  \ar[rr]^{\mathrm{H}^i(d)}  & & \mathrm{H}^i(D, T_D) \ar[r]  &  \mathrm{H}^{i+1}(T_P(-D)),}
\]   the result   follows from  the fact that $\dim \mathrm{H}^1(P, T_P(-D)) = g\cdot \dim \mathrm{H}^{g-1}(P, \mathcal{O}_P(D)) =0$ for $g=\dim P \geq 2$  and $\dim \mathrm{H}^2(P,T_P(-D)) = g\cdot \dim \mathrm{H}^{g-2}(P, \mathcal{O}_P(D)) =0$ for $g\geq 3$ (see \cite[p.~150]{MumAb}) because $\mathcal{O}_P(D)$ is non-degenerate and effective. \end{proof}

 Since a morphism of finite type separated Deligne-Mumford stacks is unramified if and only if it is injective on tangent spaces (use  \cite[Tag~0B2G]{stacks-project}), we obtain the following consequence of Proposition \ref{prop:tangent}.

\begin{corollary}\label{prop:qf0}  The morphism
$\mathcal{AH}_{g,d, \mathbb{Q}}^{\text{sm}} \to \mathcal{CP}$ is unramified. \qed
\end{corollary}

	 Since unramified morphisms of finite type are quasi-finite \cite[Tag~06PU]{stacks-project}, we obtain the following useful consequence from Corollary \ref{prop:qf0}.  
 
\begin{corollary}\label{cor:qf} The morphism
$\mathcal{AH}_{g,d, \mathbb{Q}}^{\text{sm}} \to \mathcal{CP}$ is quasi-finite. \qed
\end{corollary}

 \subsection{The second proof of Proposition \ref{prop:uniff}}\label{section33}
	Recall that  Proposition \ref{prop:uniff} says that the stack $\mathcal{AH}_{g,d,\mathbb{Z}[1/d]}^{\text{sm}}$ is  uniformisable.  In this  section we reprove this over $\mathbb{Q}$ by adding level structure to the hypersurface $\mathrm{V}(s)$ associated to $(A,P,L,s)$ in $\mathcal{AH}_{g,d,\mathbb{Q}}^{\text{sm}}$.  
	
	We first record the following well-known lemma concerned with  the action of an automorphism of an abelian variety on  its (singular) cohomology. Note that, to simplify the notation, we will omit writing $X^{\an}$ instead of $X$, so that, for example,    $\mathrm{H}^\ast(X,\mathbb{Z})$ denotes the singular $\mathbb{Z}$-cohomology of $X^{\an}$.
	
	\begin{lemma}\label{lem:abvar} Let $A$ be an abelian variety over $\mathbb{C}$, and let $\sigma:A\to A$ be an isomorphism of $\mathbb{C}$-schemes which acts   trivially on $\mathrm{H}^1(A,\mathbb{C})$. Then, there is an element $a\in A(\mathbb{C})$ such that $\sigma$ is given by translation by $a$.
	\end{lemma}
	\begin{proof}  
	By \cite[Corollary~II.1]{MumAb}, there is an element $a$ in $A(\mathbb{C})$ and an isomorphism $f:A\to A$ such that, for every $x$ in $A(\mathbb{C})$, we have that $\sigma(x) = f(x) + a$.  Since translations act trivially on cohomology and     $\sigma$  acts trivially on $\mathrm{H}^1(A,\mathbb{C})$, we see that  $f$ acts trivially on $\mathrm{H}^1(A,\mathbb{C})$. In particular,  $f-\mathrm{id}_A$ acts as zero on $\mathrm{H}^1(A,\mathbb{C})$. Since    $\mathrm{End}(A)$ injects into $\mathrm{End}(\mathrm{H}^1(A,\mathbb{C}))$   \cite[p.~175-177]{MumAb}, we conclude that   $f$ is the identity map $\mathrm{id}_A:A\to A$, as required.
	\end{proof}
	
	In fact, as Will Sawin explained to us, the automorphism group of a smooth projective variety of general type which embeds into its Albanese, acts faithfully on its (singular) cohomology. 
	
	\begin{lemma} \label{lem:sawin}
	Let $X$ be a smooth projective variety of general type over $\CC$ which embeds into an abelian variety.   Then $\mathrm{Aut}(X)$ acts faithfully on $\mathrm{H}^*(X,\mathbb{C})$.
	\end{lemma}
	\begin{proof}  Let $\sigma$ be an automorphism of $X$ such that $\sigma$ acts trivially on $\mathrm{H}^\ast(X,\mathbb{C})$. We   show that $\sigma$ is the identity.
	
	 First, by \cite[Appendix, Theorem~1]{DebarreJiangLahoz}, since $X$ is of general type, the topological Euler characteristic $e(X) \neq 0$. Since $\sigma\in\mathrm{Aut}(X)$ acts trivially on the entire cohomology $\mathrm{H}^\ast(X,\mathbb{C})$,   the Lefschetz trace formula together with $e(X)\neq 0$ implies that $\sigma$ has a fixed point (see    \cite[III.4.11.4, p.111]{SGA5}). 
	 
	 On the other hand, since $X$ embeds into an abelian variety, we have that $X$ embeds into its Albanese variety $\mathrm{Alb}_{X/\mathbb{C}}$. In particular, $\sigma$ induces an automorphism $\psi:\mathrm{Alb}_{X/\mathbb{C}}\to \mathrm{Alb}_{X/\mathbb{C}}$ and we have a commutative diagram
	\[ 
\xymatrix{ X \ar[d]_{\textrm{inclusion}} \ar[rr]^{\sigma} & & X \ar[d]^{\textrm{inclusion}}  \\ \mathrm{Alb}_{X/\mathbb{C}} \ar[rr]^{\psi} & & \mathrm{Alb}_{X/\mathbb{C}}}
	\] Since   $\sigma$ acts trivially on $\mathrm{H}^*(X,\mathbb{C})$, it acts trivially on $\mathrm{H}^1(X,\mathbb{C})$. In particular, since $\mathrm{H}^1(X,\mathbb{C}) = \mathrm{H}^1(\mathrm{Alb}_{X/\mathbb{C}},\mathbb{C})$, it follows that   $\psi$ acts trivially on $\mathrm{H}^1(\mathrm{Alb}_{X/\mathbb{C}},\mathbb{C})$.    By Lemma \ref{lem:abvar}, this implies that  $\psi$ is given by  translation on $\mathrm{Alb}_{X/\mathbb{C}}$.

	Since $\sigma$ has a fixed point, the same holds for $\psi$. However, because $\psi$ is a translation it must be the identity. By the above commutative diagram, we conclude that $\sigma$ is the identity, as required.   
	\end{proof}

	\begin{lemma} \label{lem:sawin100} Let $S$ be a finite type $\mathbb{C}$-scheme and let $X \to S$ be a smooth proper morphism whose geometric fibres are canonically polarized varieties.  Suppose that, for every $s$ in $S(\mathbb{C})$, the group $\mathrm{Aut}(X_s)$ acts faithfully on $\mathrm{H}^*(X_s,\mathbb{C})$. Then, there is an integer $\ell_0 \geq 3$ such that, for every prime number $\ell \geq \ell_0$ and every $s$ in $S(\mathbb{C})$, the action of $\mathrm{Aut}(X_s)$ on $\mathrm{H}^*(X_{s},\mathbb{Z}/\ell\mathbb{Z})$ is faithful. \end{lemma}
	
	\begin{proof}  Stratifying $S$ by finitely many locally closed subschemes, we may and do assume that $S$ is a smooth integral variety over $\mathbb{C}$. Therefore the (co)homology groups $\mathrm{H}_*(X_s,\mathbb{Z})$ and $\mathrm{H}^*(X_s,\mathbb{Z})$ are all isomorphic as $s$ varies in $S(\mathbb{C})$. Thus, since these groups are finitely generated, there is an integer $\ell_0\geq 3$ such that every prime number $\ell \geq \ell_0$ is prime to the order of the torsion in the groups above. We obtain the following identifications for every such $\ell$ and every $s$ in $S(\mathbb{C})$ from the universal coefficient theorem:
	\[\mathrm{H}^{\ast}(X_s,\mathbb{Z}) \otimes \mathbb{Z}/{\ell}^n\mathbb{Z} \cong \mathrm{H}^{\ast}(X_s,\mathbb{Z}/\ell^n\mathbb{Z})\]
	\[\mathrm{H}^\ast(X_s,\ZZ) \otimes \ZZ_{\ell} \cong \mathrm{H}^\ast(X_s,\ZZ_\ell)\]
	 Moreover, by our choice of $\ell$, the group $\mathrm{H}^{\ast}(X_s,\mathbb{Z}) \otimes \mathbb{Z}_{\ell}$ is torsion-free. 
	
	Let $s\in S(\mathbb{C})$, and suppose that $\sigma \in \mathrm{Aut}(X_s)$  acts trivially on $\mathrm{H}^\ast(X_{s},\ZZ/\ell \ZZ)$, we will show that $\sigma$ is trivial.

	 Since $\mathrm{H}^\ast(X_s,\ZZ_\ell)$ is torsion-free, we may consider the action of $\sigma$ on the free finitely generated $\mathbb{Z}_\ell$-module $\mathrm{H}^\ast(X_s,\ZZ_{\ell})$ as given by     some $\ZZ_\ell$-matrix $A$.
	   Since $A$ acts trivially on $\mathrm{H}^\ast(X_{s},\ZZ/\ell \ZZ)$, we have  that $A \bmod \ell$ is the identity matrix. As $\mathrm{Aut}(X_s)$ is finite, $A$ is semi-simple and the eigenvalues of $A \otimes \bar{\ZZ}_\ell$ are roots of unity congruent to $1 \bmod \ell$, as $A \bmod \ell$ is the identity matrix. Therefore, as $\ell >\ell_0 \geq 3$, a well-known lemma of Minkowski and Serre (see the appendix of \cite{SerreRig}, or the more general \cite[Theorem~6.7]{SilverbergZarhin}) implies that each eigenvalue of $A \otimes \bar{\ZZ}_\ell$ is equal to $1$. Therefore, since $A$ is semi-simple, it is the identity matrix. It follows that  $\sigma$ acts trivially on $\mathrm{H}^\ast(X_s,\ZZ_\ell)$, and therefore also acts trivially on $\mathrm{H}^\ast(X_s,\mathbb{C})$. By our assumption that the group $\mathrm{Aut}(X_s)$ acts faithfully on $\mathrm{H}^\ast(X_s,\mathbb{C})$, it follows that   $\sigma$ is the identity.  
	\end{proof}

We want to identify a substack $\mathcal{CP}^+ \subset \mathcal{CP}$ which parametrizes those smooth proper canonically polarized varieties  which embed into some abelian variety. Equivalently, $\mathcal{CP}^+$ should parametrize exactly those smooth proper canonically polarized varieties which embed into their Albanese variety. Thus, if $\mathcal{U} \to \mathcal{CP}$ is the universal canonically polarized scheme, then $\mathcal{CP}^+$ can be realized as the (open) locus in $\mathcal{CP}$ where $\mathcal{U} \to \mathrm{Alb}_{\mathcal{U}/\mathcal{CP}}$ is a closed immersion (see \cite[9.6.1]{EGAIVII}). In particular, $\mathcal{CP}^+$ is a separated Deligne-Mumford stack over $\mathbb{Q}$.

\begin{theorem}\label{thm:cpplus}
The stack $\mathcal{CP}^+$ is uniformisable.
\end{theorem}

		\begin{proof} For $h\in \mathbb{Q}[t]$,  let $\mathcal{CP}^+_h =\mathcal{CP}^+ \cap \mathcal{CP}_h$ be the stack of canonically polarized varieties $X$ with Hilbert polynomial $h$ and which embed into their Albanese. It suffices to show that $\mathcal{CP}^+_h$ is uniformisable. We will use level $\ell$-structure on the \emph{entire} cohomology (compare with \cite{PoppIII} and \cite{JLlevel}).

  Since $\mathcal{CP}^+_h$ is of finite type over $\mathbb{Q}$, Lemma \ref{lem:sawin} and Lemma \ref{lem:sawin100}  imply that we may choose a prime number $\ell$ such that, for every $X$ in $\mathcal{CP}^+_h(\mathbb{C})$, the action of $\mathrm{Aut}(X)$ on $\mathrm{H}^*(X,\mathbb{Z}/\ell \mathbb{Z})$ is faithful. 
	  Define $(\mathcal{CP}^+_h)^{[\ell]}$ to be the stack over $\mathbb{Q}$ whose objects are tuples $(f:X\to S,\phi_1,\ldots, \phi_{2n})$, where $n:=\deg h$, the morphism $f:X\to S$ is in $\mathcal{CP}^+_h$ and, for every $i=1,\ldots, 2n$, the morphism $\phi_i: R^i f_\ast \mathbb{Z}/\ell\mathbb{Z} \cong (\mathbb{Z}/\ell\mathbb{Z})_S^{b_i}$ is an isomorphism of group schemes over $S$. Here $b_i $ is the $i$-th Betti number of (a geometric fibre of) $X\to S$.  Note that $(\mathcal{CP}^+_h)^{[\ell]}\to  \mathcal{CP}^+_h$ is finite and \'etale, so that $(\mathcal{CP}^+_h)^{[\ell]}$ is a finite type separated algebraic stack over $\mathbb{Q}$. Moreover, since the action of $\mathrm{Aut}(X)$ on $\mathrm{H}^*(X,\mathbb{Z}/\ell \mathbb{Z})$ is faithful for every $\mathbb{C}$-point $X$ of $\mathcal{CP}_h^+$, the stack  $(\mathcal{CP}^+_h)^{[\ell]}$ is an algebraic space, as required.
	\end{proof}
	
	\begin{proof}[Second proof of Proposition \ref{prop:uniff}]
	 The morphism  $\mathcal{AH}_{g,d,\mathbb{Q}}^{\text{sm}} \to \mathcal{CP}$ is representable  (Lemma \ref{lem:reps}), and factors over $\mathcal{CP}^+$. Since $\mathcal{CP}^+$ is uniformisable (Theorem \ref{thm:cpplus}), it follows that $\mathcal{AH}_{g,d,\mathbb{Q}}^{\text{sm}}$ is uniformisable. 
	\end{proof}

	\section{Proof of Main Theorem}
For $\mathcal{G}$ a small groupoid, we let $\pi_0(\mathcal{G})$ denote the set of isomorphism classes of objects of $\mathcal{G}$.
	The following finiteness for the stack of smooth abelian hypersurfaces is the starting point of this paper. This theorem is obtained by combining the results of \cite{Faltings2, lawrence2020shafarevich} with finiteness results for torsors under an abelian scheme; see  \cite[Theorem~6.5]{JLM} for a detailed proof.
 
	\begin{theorem} \label{thm:sawinhyp} 
	If $d \geq 1$,  $g \geq 4$,    and $S$ is a finite set of finite places of a number field $K$, then  $\pi_0(\mathcal{AH}_{g,d}^{\text{sm}}(\OO_{K,S}))$ is finite. \qed
	\end{theorem}
	
	Roughly speaking, Theorem \ref{thm:sawinhyp} says that our desired finiteness result holds over number rings, even when we vary the ambient abelian variety. Now, to prove Theorem \ref{thm1}, it suffices to show that, for every  arithmetic ring  $R$, the set $\pi_0(\mathcal{AH}_{g,d}^{\text{sm}}(R))$ is finite. In fact, this is   \emph{stronger} than the conclusion of Theorem \ref{thm1}.

\begin{theorem}\label{thm:main}	If $d \geq 1$,  $g \geq 4$,    and $R$ is an arithmetic ring, then  $\pi_0(\mathcal{AH}_{g,d}^{\text{sm}}(R))$ is finite.
\end{theorem}
\begin{proof}
Let $$U\to  \mathcal{AH}_{g,d, \mathbb{Q}}^{\text{sm}}$$ be a finite \'etale surjective  morphism with $U$ a finite type separated algebraic space over $\mathbb{Q}$; such an algebraic space exists by  Proposition \ref{prop:uniff}. Note that $U$ is a quasi-projective scheme as the induced morphism $U\to \mathcal{CP}$ is quasi-finite and separated (Corollary \ref{cor:qf}) and the coarse space of $\mathcal{CP}_h$ is quasi-projective for every $h\in \mathbb{Q}[t]$ by Viehweg's theorem  \cite[Theorem~3]{Viehweg06}.
 
By Theorem \ref{thm:sawinhyp}, the stack $\mathcal{AH}_{g,d, \Qbar}^{\text{sm}}$ is arithmetically hyperbolic over $\Qbar$ (as defined in \cite[Definition~4.1]{JLalg}). Since $$U_{\Qbar}\to \mathcal{AH}_{g,d, \Qbar}^{\text{sm}}$$ is quasi-finite, it follows that $U_{\Qbar}$ is arithmetically hyperbolic over $\Qbar$ (see \cite[Proposition~4.17]{JLalg}).  Since the composed morphism \[U_{\Qbar} \to  \mathcal{AH}_{g,d, \Qbar}^{\text{sm}} \to \mathcal{CP}\] is quasi-finite  (Corollary \ref{cor:qf}), we have that $U_{\Qbar}$ satisfies the Persistence Conjecture \cite[Theorem~1.4]{JSZ}, so that, for every algebraically closed field $k$ of characteristic zero, the variety $U_k$ is arithmetically hyperbolic over $k$. The stacky Chevalley-Weil theorem \cite[Theorem~5.1]{JLalg} allows us to conclude that $  \mathcal{AH}_{g,d, k}^{\text{sm}}$ is arithmetically hyperbolic over $k$ for every such field.
 
  Let $R$ be a $\mathbb{Z}$-finitely generated normal integral domain of characteristic zero, and define $k:=\overline{\mathrm{Frac}(R)}$. Since  $\mathcal{AH}_{g,d}^{\text{sm}}$ has finite diagonal (Proposition \ref{prop:basic_properties}) and $\mathcal{AH}_{g,d,k}^{\text{sm}}$  is arithmetically hyperbolic over $k$, by applying  \cite[Theorem~4.23]{JLalg}, we obtain that  the set  $\pi_0(\mathcal{AH}_{g,d}^{\text{sm}}(R))$ is finite, as required.  
\end{proof}

\begin{proof}[Proof of Theorem \ref{thm1}] 
Let $A$ be an abelian scheme over $R$ of dimension $g$ and  let $\mathcal{L} \in \Pic(A)$  be an ample line bundle  of degree $d$. Define $S$ to be the set of smooth hypersurfaces $H \subset A$ with $\mathcal{O}_A(H) \cong \mathcal{L}$.  
Note that  there is a map $S \to \pi_0(\mathcal{AH}_{g,d}^{\text{sm}}(R))$ which sends a hypersurface $H \subset A$ to the isomorphism class of the tuple \[(A,A,\mathcal{O}_A(H), s_H: \mathcal{O}_A \to \mathcal{O}_A(H)).\] The   set $\pi_0(\mathcal{AH}_{g,d}^{\text{sm}}(R))$ is finite by the arithmetic hyperbolicity of $\mathcal{AH}_{g,d}^{\text{sm}}$ over the algebraic closure of $\text{Frac}(R)$ (see Theorem \ref{thm:main}).  Thus, to prove the desired finiteness of $S$, it suffices to show that the above map has finite fibers. However, the fiber of the map $S\to\pi_0(\mathcal{AH}_{g,d}^{\text{sm}}(R))$ over a given $(A,A,\mathcal{L},s)$  consists of those hypersurfaces $H' \in S$ which appear as the images of $H$ under an $R$-isomorphism  $f: A \to A$ of schemes which preserves $\mathcal{L}$. To see this set is finite, we invoke  the  following well-known finiteness  statement:     if $X$ is a smooth projective variety of Kodaira dimension at least zero (e.g., $X$ is an abelian variety) and $E$ is an ample line bundle on $X$, then the group of automorphisms of $X$ fixing $\mathcal{O}(E)$ is finite. (This finiteness  is proven as follows.  Let $G$ be the group scheme of automorphisms of $X$ fixing $\mathcal{O}(E)$.  To prove that $G$ is finite,    we may and do assume that $E$ is very ample. Then, any automorphism $\sigma:X\to X$ with $\sigma^\ast \mathcal{O}(E) \cong \mathcal{O}(E)$ extends to an automorphism of $\mathbb{P}(\mathrm{H}^0(X,\mathcal{O}(E)))$. Therefore, $G$ is an affine finite type group scheme. Since the Kodaira dimension of $X$ is nonnegative, by  Matsusaka-Mumford's theorem \cite[Theorem~2]{MatMum}, the group scheme $G$ is proper, hence finite.) This concludes the proof.
\end{proof}

  \bibliography{refsci}{}
\bibliographystyle{alpha}

	\end{document}